\documentclass[12pt]{article}
\newcommand{\authorfootnotes}{\renewcommand\thefootnote{\@fnsymbol\c@footnote}}%
\usepackage{amssymb,amsmath, amsfonts, amsthm}

\usepackage{graphicx}
\usepackage{tikz}
\usepackage{enumerate}
\usetikzlibrary{graphs}

\tikzstyle{vertex}=[circle, draw, inner sep=0pt, minimum size=6pt]

\usetikzlibrary{arrows}

\newtheorem{thm}{Theorem}[section]
\newtheorem{lem}[thm]{Lemma}
\newtheorem{cor}[thm]{Corollary}
\newtheorem{prop}[thm]{Proposition}
\newtheorem{prob}{Problem}
\newtheorem{conj}[thm]{Conjecture}

\newtheorem{obs}[thm]{Observation}

\newcommand{\D}{{\rm DOM}}
\newcommand{\dom}{{\rm dom}}
\newcommand{\Dt}{{\rm DOM}_t}
\newcommand{\dt}{{\rm dom}_t}
\newcommand{\gt}{\gamma_t}

\newcommand{\pn}{{\rm pn}}

\begin{document}

\title{Orientable total domination in graphs}

\author{
Sarah E. Anderson$^{a}$
\and
Tanja Dravec$^{b,c}$
\and
Daniel Johnston$^{d}$
\and
Kirsti Kuenzel$^{d}$\\
}

\date{\today}

\maketitle

\begin{center}
$^a$ Department of Mathematics, University of St. Thomas, St. Paul, MN, USA\\
$^b$ Faculty of Natural Sciences and Mathematics, University of Maribor, Slovenia\\
$^c$ Institute of Mathematics, Physics and Mechanics, Ljubljana, Slovenia\\
$^d$ Department of Mathematics, Trinity College, Hartford, CT, USA\\
\end{center}

\begin{abstract} 
Given a directed graph $D$, a set $S \subseteq V(D)$ is a total dominating set of $D$ if each vertex in $D$ has an in-neighbor in $S$. The total domination number of $D$, denoted $\gt(D)$, is the minimum cardinality among all total dominating sets of $D$. Given an undirected graph $G$, we study the maximum and minimum total domination numbers among all orientations of $G$. That is, we study the upper (lower) orientable domination number of $G$, $\Dt(G)$ (or $\dt(G)$), which is the largest (or smallest) total domination number over all orientations of $G$. We characterize those graphs with $\Dt(G) = \dt(G)$ when the girth is at least $7$ as well as those graphs with $\dt(G) = |V(G)|-1$. We also consider how these parameters are effected by removing a vertex from $G$, give exact values of $\Dt(K_{m,n})$ and $\dt(K_{m,n})$ and bound these parameters when $G$ is a grid graph.
\end{abstract}


\noindent
{\bf AMS Subj.\ Class.\ (2020)}: 05C20, 05C69, 05C76.

\maketitle

\section{Introduction}
Given an undirected graph $G$, an orientation of $G$ is any digraph where each edge in $G$ is directed in one of two possible ways. If $D$ is a directed graph, a set $S\subseteq V(D)$ is a dominating set (or total dominating set) of $D$ if each vertex of $V(D)-S$ (or $V(D)$) has an in-neighbor in $S$. The minimum cardinality of a dominating set (or total dominating set) of $D$ is denoted $\gamma(D)$ (or $\gt(D)$). Of particular interest are two  recent papers by \cite{AB-23, CH-12} in which they studied the \emph{orientable domination number} of an undirected graph $G$  defined to be $\D(G) = \max\{\gamma(D): D \text{ is an orientation of $G$}\}$ (although this is referred to as \emph{directed domination} in \cite{CH-12}). In both papers,  it was noted that replacing $\max$ with $\min$ in this definition is uninteresting as it would describe the usual domination number of $G$. In this paper, we consider the total domination version of this parameter. That is, given an undirected graph $G$,  the {\em upper orientable total domination number} is defined to be 
$$\Dt(G) = \max\{\gamma_t(D): \, D\textrm{ is a valid orientation of }G\},$$ and the {\em lower orientable total domination number} is defined to be 
$$\dt(G) = \min\{\gamma_t(D): \, D\textrm{ is a valid orientation of }G\}.$$ Notice that given an orientation $D$ of $G$, $\gt(D)$ may not exist unless every vertex in $D$ has at least one edge directed towards it. For this reason, the lower orientable total domination number is not always $\gt(G)$. Furthermore, determining the total domination number of directed graphs is quite difficult (see~\cite{AJV-2007, Hao-2017}).

$\Dt(G)$ and $\dt(G)$ were first defined in \cite{HLY-99, Lisa} at the end of the twentieth century where they were referred to as open orientable domination. We choose to refer to these parameters in terms of total domination as their definitions are the direct translation of total dominating sets in directed graphs. In~\cite{HLY-99, Lisa} it was proved that $\dt(G)$ and $\Dt(G)$ exist if and only if every component of $G$ is not a tree. It was also shown that $\Dt(G)=|V(G)|=\dt(G)$ if and only if $G$ is a cycle. Moreover in these two papers the lower and upper orientable total domination number was studied for complete graphs, although  the exact value for $\Dt(K_n)$ is still not known. The upper and lower orientable total domination number was also studied  for complete multipartite graphs~\cite{Lai-2001}, where the exact values were obtained for $\dt(K_{n_1,\ldots ,n_k})$ and $\Dt(K_{n_1,\ldots ,n_k})$ when $k \geq 3$. However,  their results do not hold for $k=2$. 

In this paper, we first study general properties of $\Dt(G)$ and $\dt(G)$ such as when these two graph parameters are equal, and how $\Dt(G)$ and $\dt(G)$ are effected by local changes to a graph. We then 
provide bounds on $\Dt(K_n)$ and give exact values of $\Dt(K_{m,n})$ and $\dt(K_{m,n})$. We also provide bounds for both parameters with respect to grid graphs. Finally, we characterize all graphs where $\dt(G) = |V(G)|-1$. 

The remainder of this paper is organized as follows. In Section~\ref{sec:def}, we provide useful definitions and previous results that are used throughput the paper. In Section~\ref{sec:gen}, we characterize the graphs with large enough girth where $\Dt(G) = \dt(G)$. We also consider the effect on $\Dt(G)$ and $\dt(G)$ upon removal of a vertex from $G$. Lastly, we compare $\D(G)$ and $\Dt(G)$. In Section~\ref{sec:complete}, we bound $\Dt(K_n)$, and we give exact values of both $\Dt(K_{m, n})$ and $\dt(K_{m,n})$. We consider $\Dt(P_m\Box P_n)$ and $\dt(P_m\Box P_n)$ in Section~\ref{sec:prod}. Lastly, in Section~\ref{sec:extreme}, we provide a characterization of all graphs $G$ where $\dt(G) = |V(G)|-1$.

\subsection{Notation and preliminaries}\label{sec:def}
Let $D=(V(D),A(D))$ be a digraph. A vertex $u$ is an \emph{in-neighbor} of $v$ if $(u,v) \in A(D)$ and an \emph{out-neighbor} of $v$ if $(v,u) \in A(D)$. The \emph{open out-neighborhood} of $v$ is the set of out-neighbors of $v$ and is denoted by $N^+_D(v)$. The \emph{closed out-neighborhood} of $v$ is the set $N^+_D[v]$ defined by $N^+_D[v]=N^+_D(v)\cup \{v\}$.  In a similar manner one defines the \emph{open in-neighborhood} $N^-_D(v)$ of $v$ and the \emph{closed in-neighborhood} $N^-_D[v]$ of $v$. The \emph{in-degree} of $v$ is the number $|N^-_D(v)|$ and the \emph{out-degree} of $v$ is $|N^+_D(v)|$. The subgraph of $D$ induced by a set $S \subseteq V(D)$ is denoted by $D[S]$. For $k \ge 1$ an integer, we let $[k]$ denote the set $\{1,\ldots,k\}$.
If the digraph $D$ is clear from the context, then we may omit the subscript $D$ from the above notations.

A vertex $x \in V(D)$ \emph{ dominates} a vertex $y$ if $y\in N_D^+(x)$, and we then also say that $y$ is \emph{ dominated} by $x$. If $S \subseteq V(D)$, then $y \in V(D)$ is \emph{ dominated} by $S$ if there exists $x\in S$ that  dominates $y$. A set $S$ is a \emph{total dominating set} (or shortly, a \emph{TD}-\emph{set}) of $D$ if every vertex in $D$ is dominated by $S$. Note that a digraph $D$ has a total dominating set if and only if $|N_D^-(v)| \geq 1$ for any $v \in V(D)$.   The minimum cardinality of a TD-set of $D$ is the \emph{total domination number} of $D$, denoted $\gamma_t(D)$. A total dominating set $S$ of $D$ with $|S|=\gt(D)$ will be called a $\gt$-set of $D$.  

For a set $S \subseteq V(D)$ and a vertex $v \in S$, the \emph{open $S$-private neighborhood} of $v$, denoted by $\pn(v,S)$, is the set of vertices that are in the open out-neighborhood of $v$ but not in the open out-neighborhood of the set $S \setminus \{v\}$. Equivalently, $\pn(v,S) = \{w \in V(D) \, \colon N^-(w) \cap S = \{v\}\}$.  If $S$ is a $\gt$-set of a digraph $D$, then $\pn(v,S) \neq \emptyset$ holds for any $v \in S$.

Let $G$ be an undirected graph. An {\em orientation of $G$} is a digraph in which every edge from $G$ is directed in one of the two possible directions. Formally, an orientation of $G$ is defined by a mapping $f:E(G)\to V(G)\times V(G)$, such that if $uv\in E(G)$, then $f(uv)\in \{(u,v),(v,u)\}$. We denote this orientation of $G$ by $G_f$, while we refer to $f$ as the {\em orienting mapping}. In this paper, we use the same definition as that used in \cite{AB-23} where $\D(G) = \max\{\gamma(G_f):\, f\textrm{ is an orienting mapping of }G\}$. If $G$ is a graph, then each edge can be oriented in two different directions. Thus, $G$ is the underlying graph of many different digraphs, where some of them can have large but others small total domination numbers. Moreover, the total domination number of a digraph can be computed only if the orientation of the digraph is valid, i.e.\ when each vertex of the digraph has in-degree at least $1$. Thus we define the {\em upper orientable total domination number} as
$$\Dt(G) = \max\{\gamma_t(G_f): \, f\textrm{ is a valid orienting mapping of }G\},$$ and the {\em lower orientable total domination number} as
$$\dt(G) = \min\{\gamma_t(G_f): \, f\textrm{ is a valid orienting mapping of }G\}.$$
Both invariants were first defined in~\cite{HLY-99, Lisa}, where it was proved that $\Dt(G)$ and $\dt(G)$ exists if and only if every  component of $G$ is not a tree. Clearly, if $T$ is a tree, then in any orientation $f$ of $T$, $T_f$ contains a vertex $v$ with $|N^-_{T_f}(v)|=0$. Thus, we will only consider graphs in which every  component contains a cycle.  Moreover, since the total domination number of a digraph $D$ is the sum of the total domination numbers of the components in $D$, we only consider connected graphs that are not a tree. We let $\mathcal{C}$ represent the class of all connected graphs that contain a cycle.

As stated earlier,  $\dom(G)$ is not really interesting, as one can easily deduce that $\dom(G)=\gamma(G)$ holds for any graph $G$. Indeed, if $S$ is a minimum dominating set of $G$, then we can orient the edges of $G$ such that for every edge $e=xy$ with one end-vertex $x \in S$ and the other end-vertex $y \notin S$, $(x,y)$ is the edge of the corresponding digraph $D$. Thus, $S$ is a dominating set of $D$. This is not the case when studying orientable total domination of a graph $G$, as any vertex in a valid orientation must have in-degree at least $1$. Thus, $\gt(G) \leq \dt(G)$ holds for any graph $G \in \mathcal{C}.$ Moreover the difference between $\dt(G)$ and $\gt(G)$ can be arbitrarly large. For example, $\gt(C_n)\leq \lfloor \frac{n}{2} \rfloor +2$ but $\dt(C_n)=n$~\cite{Lisa}.

We use similar terminology as presented above for an undirected graph $G=(V(G),E(G))$, with $V(G)$ being the vertex set, $E(G)$ the edge set of $G$ and $n(G)=|V(G)|$ the order of $G$. The \emph{open neighborhood} of a vertex $v$ in $G$ is $N_G(v) = \{u \in V(G) \, \colon \, uv \in E(G)\}$ and the \emph{closed neighborhood of $v$} is $N_G[v] = \{v\} \cup N_G(v)$. We denote the degree of $v$ in $G$ by $\deg_G(v)$, and so $\deg_G(v) = |N_G(v)|$. For a set $X \subseteq V(G)$ we denote by $N_G[X]=\cup_{x \in X} N_G[x]$ the closed neighborhood of $X$. For two vertices $x,y \in V(G)$ the \emph{distance} between $x$ and $y$ in $G$, $d_G(x,y)$, is the length of a shortest $x,y$-path in $G$. The subgraph of $G$ induced by a set $D \subseteq V(G)$ is denoted by $G[D]$. Moreover, a vertex of degree $1$ is called a {\em leaf} and the set of all leaves of $G$ is denoted by $L(G)$. The {\em girth} of $G$, $g(G),$ is the cardinality of a smallest cycle in $G$. The  {\em Cartesian product} $G\Box H$ of graphs $G$ and $H$ is defined as the graph with $V(G\Box H)=V(G)\times V(H)$, and $(g,h)(g',h')\in E(G\Box H)$ if either ($g=g'$ and $hh\in E(H)$) or ($gg'\in E(G)$ and $h=h'$). A {\em grid} is the Cartesian product of two paths and a {\em ladder} is the Cartesian product of a path and $K_2$.

\section{General bounds}\label{sec:gen}

We start the section with some general observations that will be used several times in the paper. Since each vertex $v$ of a total dominating set $S$ in a digraph $D$ must be dominated by $S$ (i.e.\ it has an in-neighbor in $S$), we have the following.

\begin{obs}\cite{HLY-99,Lisa}
\label{tdsetcycle}
Let $G$ be a graph and $G_f$ be an orientation of $G$. If $S$ is a TD-set of $G_f$, then every component of $G[S]$ contains at least one cycle.
\end{obs}

The above observation immediately gives us a trivial lower bound for $\dt(G)$ and $\Dt(G)$, as stated in the following observation. Note that one can achieve the lower bound for $\Dt(G)$ by taking a longest induced cycle in $G$, orienting the edges of said cycle in an oriented cycle, and orienting all remaining edges in $G$ away from the cycle.

\begin{obs} 
\label{DOMgirth}
For any $G \in \mathcal{C}$, $\dt(G) \geq g(G)$ and  $\Dt(G) \ge r$ where $r$ is the length of the longest induced cycle in $G$. 
\end{obs}

We already know that a connected graph has a valid orientation only when it contains a cycle.  Next, we obtain exact values of $\Dt(G)$ and $\dt(G)$ for any unicyclic graph $G$. Recall that for any graph $G$, $L(G)$ represents the set of leaves in $G$.

\begin{prop}\label{unicyclic}
If $G$ is a connected unicyclic graph, then $\Dt(G)=\dt(G)=|V(G)|-|L(G)|$. Moreover $V(G)-L(G)$ is a total dominating set of $G_f$, for any valid orienting mapping $f$ of $G$.
\end{prop}  
\begin{proof}
Let $f$ be an arbitrary valid orienting mapping of $G$. Since all vertices of $G_f$ have in-degree at least 1, edges of the cycle are oriented cyclically. Moreover, for any other edge $uv \in E(G)$ it holds that $(u,v) \in A(G_f)$ if the distance from $u$ to any vertex on the cycle is less than the distance from $v$ to any vertex on the cycle by Observation~\ref{tdsetcycle}.
Thus, we have just two possible orientations of $G$, depending on which direction on the cycle we choose. In both cases $\gt(G_f)=|V(G)|-|L(G)|$ and hence $\dt(G)=\Dt(G) =|V(G)|-|L(G)|$.
\end{proof}

If the girth is large enough we can characterize the graphs achieving the bound from Observation~\ref{DOMgirth}. Indirectly, we have also given a partial characterization of the graphs for which $\Dt(G) = \dt(G)$. 

\begin{thm}
\label{DOMunicyclic}
Let $G$ be a graph with $g(G) \ge 7$. Then the following statements are equivalent.
\begin{itemize}
\item(i) $\Dt(G)=g(G)$ or $\dt(G)=g(G)$;
\item(ii) $\Dt(G) =\dt(G) =  g(G)$;
\item(iii) $G$ is unicyclic and $G$ is obtained from a cycle by appending any number of leaves to each vertex of the cycle.
\end{itemize} 
\end{thm}
\begin{proof}
We first prove that $(iii) \Rightarrow (i),(ii)$. Thus let $G$ be unicyclic graph obtained from a cycle by appending any number of leaves to each vertex of the cycle. By Proposition~\ref{unicyclic}, $\dt(G)=\Dt(G)=|V(G)|-|L(G)|=g(G)$ which proves $(i)$ and $(ii)$. Moreover implication $(ii) \Rightarrow (i)$ is trivial. To complete the proof it remains to show $(i) \Rightarrow (iii)$.

Assume first that $\dt(G) = g(G)$. Let $G_f$ be an orientation of $G$ such that $\gamma_t(G_f) = \dt(G)$ and let $S$ be a $\gt$-set of $G_f$. By Observation \ref{tdsetcycle}, $S$ must contain a cycle, call it $C$. Since $\dt(G) = g(G)$, $C$ has length $g(G)$ and is a $\gamma_t$-set. Thus, all vertices in $G_f$ must be dominated by $C$. Also, if $v \in V(G_f)\setminus C$, then $v$ is adjacent to exactly one vertex on $C$; otherwise, if $v$ has two neighbors on $C$, then there is a cycle of length less than $g(G)\ge 7$. In addition, if $v, w \in V(G_f)\setminus C$, then $v$ and $w$ are not adjacent; otherwise, we again form a cycle of length less than $g(G)$.  Thus, $G$ is obtained from a cycle by appending any number of leaves to each vertex of the cycle. 

Finally, suppose $\Dt(G)=g(G)$. Then it follows from Observation~\ref{DOMgirth} that $\dt(G)=g(G)$ and hence the previous paragraph implies that $G$ is obtained from a cycle by appending any number of leaves to each vertex of the cycle

\end{proof}

We point out that it is not clear what $G$ must look like if $\Dt(G) = \dt(G)$ when $\dt(G) \ne g(G)$ or $g(G) <7$.




Next, we consider how the upper and lower orientable total domination numbers of $G$ are effected upon removal of a vertex in $G$. The following lemma was considered in \cite{Lisa}. 
\begin{lem}[\cite{Lisa}, Lemma 11]\label{lem:Lisa} Let $G$ be a graph and let $v$ be any vertex of $G$. If $\Dt(G-v)$ is defined, then 
\[\Dt(G) \ge \Dt(G-v) \ge \Dt(G) -1.\]
\end{lem}

Unfortunately, there is a flaw in the proof of the lower bound in the above result. However, the upper bound holds and we restate it as follows.

\begin{lem}[\cite{Lisa}, Lemma 11] If $G$ and $H$ are in $\mathcal{C}$ where $H$ is an induced subgraph of $G$, then $\Dt(H) \le \Dt(G)$. 
\end{lem}


Although we were unable to prove the lower bound in Lemma~\ref{lem:Lisa}, we still believe that the result is correct. Therefore, we state this as a conjecture.

\begin{conj}\label{conj:G-v}
Let $G$ be a graph in $\mathcal{C}$ and $v \in V(G)$ such that $G-v \in \mathcal{C}$. Then $\Dt(G) \leq \Dt(G-v)+1$. 
\end{conj}  

Note that the same result holds also for $\D(G)$. Moreover if $H$ is a spanning subgraph of $G$, then $\D(G) \le \D(H)$~\cite{CH-12}. The result does not hold in terms of total domination. There exist graphs $G$ with spanning subgraphs $H$ and $\Dt(H) \leq \Dt(G)$. Let $G$ be the graph with $V(G)=\{x,y,z,u,v\}$ and $E(G)=\{xy,xz,yz,yu,zv,uv\}$ and let $H=G-uv$. Let $D$ be the digraph obtained from $G$ by orienting the edges of $G$ as $(x,y),(y,z),(z,x),(y,u),(v,z),(u,v)$. Since $\{x,y,z\}$ is not a total dominating set of $D$, it follows by Observation~\ref{tdsetcycle} that $\Dt(G) \geq \gt(D) \geq 4$. On the other hand, since $H$ is unicyclic, $\Dt(H)=|V(H)|-|L(H)|=3$, by Proposition~\ref{unicyclic}.

For the lower orientable total domination number we can prove the following expected result.

\begin{thm}\label{spanning}
If $H \in \mathcal{C}$ is a spanning subgraph of $G\in \mathcal{C}$, then $\dt(G) \leq \dt(H)$.
\end{thm}
\begin{proof}
Let $f$ be an orienting mapping of $H$ with $\gt(H_f)=\dt(H)$. Let $f'$ be an orienting mapping of $G$ obtained from $f$ by directing edges of $E(G)\setminus E(H)$ arbitrarily. Since adding arcs cannot
increase the total domination number, we have that $\dt(G) \leq \gt(G_{f'}) \leq \gt(H_{f})=\dt(H)$.
\end{proof}

Unlike in the case of the upper orientable total domination number, we know exactly what happens with respect to the  lower orientable total domination number if a vertex is removed from the graph. Let $G$ be isomorphic to the wheel $W_n$ ($W_n$ is the graph obtained from a cycle $C_n$ by adding a universal vertex) and let $v$ be the universal vertex of $W_n$. Then $\dt(W_n)=3$ and  $\dt(W_n-v)=\dt(C_n)=n$. Thus, the  lower orientable total domination number of $G-v$ can be arbitrarily larger than $\dt(G)$. Clearly it is also possible that $\dt(G) > \dt(G-v)$. For example, let $H$ be the graph obtained from $C_n$ for $n\ge 3$ by appending any positive number of leaves. Let $G$ be the graph obtained from $H$ by appending exactly one leaf, call it $v$, to any leaf in $H$. Hence, $H= G-v$ and 
\[\dt(G) = n+1 > n = \dt(G-v).\]
 However, if $\dt(G) > \dt(G-v$), then  $\dt(G)- \dt(G-v)$ cannot be very large.

\begin{lem}
Let $G \in \mathcal{C}$ and let $v \in V(G)$ such that $G-v \in \mathcal{C}$. Then $\dt(G) \leq \dt(G-v)+1$.
\end{lem}

\begin{proof}
Let $f$ be an orienting mapping of $G-v$ such that $\dt(G-v)=\gt((G-v)_f)$. Let $f'$ be the orienting mapping of $G$ obtained from $f$ by orienting all edges incident with $v$ towards $v$. Then $\dt(G) \leq \gt(G_{f'}) \leq \gt((G-v)_f)+1=\dt(G-v)+1$.
\end{proof}

One final note about Conjecture~\ref{conj:G-v} is that we can prove the following weaker version.

\begin{lem} For any $G \in \mathcal{C}$ where $G$ is not a cycle, there exists a vertex $v\in V(G)$ such that $\Dt(G) \le \Dt(G-v)+1$. 
\end{lem}

\begin{proof}
Let $G_f$ be an orientation of $G$ such that $\gt(G_f) = \Dt(G)$. Let $S$ be a $\gt(G_f)$-set. Since $G$ is not a cycle, $V(G) - S \ne \emptyset$. Pick a vertex $v \in V(G) - S$. Note that each component of $G-v$ is in $\mathcal{C}$ by Observation~\ref{tdsetcycle} as $S$ is assumed to be a total dominating set of $G_f$. Moreover, since $S$ is a total dominating set of $G_f$, for each vertex $x \in V(G_f)-\{v\}$, there exists $s \in S$ such that $(s,x) \in A(G_f)$. Thus every vertex in $G_f-v$ has in-degree at least $1$ and hence there exists a minimum total dominating set $T$ of $G_f-v$. For any in-neighbor $u$ of $v$  in $G_f$, $T\cup \{u\}$ is a total dominating set of $G_f$. Therefore, \[\gt(G_f) \le \gt(G_f-v)+1 \le \Dt(G-v)+1.\]
\end{proof}


Since the domination number can be computed for any orientation $f$ of $G$, but the total domination number can only be obtained  for valid orientations, it is possible that $\D(G)=\gamma(G_f)$, where $f$ is an orientation that contains a vertex with in-degree 0. Thus, $\D(G)$ is not necessarily smaller than $\Dt(G)$. For example, consider the graph $G$  obtained from $K_3=uvx$ by appending exactly two leaves to $u$ and exactly two leaves to $x$. If $D$ is any orientation of $G$ where the leaves of $G$ have in-degree $0$, then 
 $5 = \D(G)$. However,  $\Dt(G) = 3$ by Proposition~\ref{unicyclic}. On the other hand, $\Dt(K_{2, n}) = 4 < n = \D(K_{2, n})$, which implies that the upper orientable total domination number can be arbitrarily smaller than the upper orientable domination number. This leads us to the following observation.

\begin{obs}\label{lem:DOM} For any $G \in \mathcal{C}$, if $G$ has an orienting mapping $f$ such that $\gamma(G_f) = \D(G)$ and every vertex in $G_f$ has in-degree at least $1$, then $\D(G) \le \Dt(G)$.
\end{obs}

\begin{proof}
Let $G_f$ be such an orientation of $G$ and take any minimum total dominating set $S$ of $G_f$. Then $S$ is a dominating set of $G_f$ which implies 
\[\D(G) \le |S| =\gt(G_f) \le \Dt(G).\]
\end{proof}

\section{Complete graphs and complete bipartite graphs}\label{sec:complete}

The upper and lower orientable total domination numbers of complete graphs were studied  in~\cite{Lisa}. It was proved that $\Dt(K_{n+1}) \in \{\Dt(K_n),\Dt(K_n)+1\}$ and if $\Dt(K_{n+1})=\Dt(K_n)+1$, then $\Dt(K_i)=\Dt(K_n)+1$ for any $i \in \{n+1,n+2,\ldots , 2n+2\}$. In the same paper it was also proved that for any $k$, where $3=\dt(K_n) \leq k \leq \Dt(K_n)$ there exists an orienting mapping $f$ of $K_n$ such that $\gt((K_n)_f)=k$. They also pose a conjecture that the last statement holds for any graph $G$.

\begin{conj}\cite{Lisa}\label{conj1}
Let $G \in \mathcal{C}$ and let $c$ be any integer such that $\dt(G) \leq c \leq \Dt(G)$. Then there exists an orientation $f$ of $G$ such that $\gt(G_f)=c$. 
\end{conj} 

Erd\H os~\cite{Erdos} first studied the  upper orientable domination number of complete graphs ~\cite{Erdos}, presenting both lower and upper bounds.  His bounds were improved by Szekeres and Szekeres to $\log_2n - 2\log_2(\log_2n) \leq \D(K_n) \leq \log_2n - \log_2(\log_2n)+2$ \cite{SS-65}, but the exact values are still not known. We use these bounds and the following to obtain the bounds for the upper orientable total domination number of complete graphs. Recall that a tournament with $n$ vertices is an orientation of $K_n$.
\begin{thm}[\cite{SS-65}]\label{thm:SS} If $T$ is a tournament with $n$ vertices, then $\gamma(T) \le \log_2 n - \log_2(\log_2 n) + 2$.
\end{thm}

\begin{thm}\label{complete} For $n\ge 3$,
 \[\log_2n - 2\log_2(\log_2 n) \le \Dt(K_n) \le \log_2 n - \log_2(\log_2 n) + 4.\]
\end{thm}

\begin{proof}
For the lower bound, any orientation of $K_n$ that contains at least one vertex with in-degree $0$ has domination number $1$ which is strictly less than $\D(K_n)$. Thus, any orientation $f$ such that $\gamma((K_n)_f) = \D(K_n)$ satisfies that every vertex of $(K_n)_f$ has in-degree at least $1$. By Observation~\ref{lem:DOM}, $\log_2n - 2\log_2(\log_2 n) \le \D(K_n) \le \Dt(K_n)$. For the upper bound, let $(K_n)_f$ be any valid orientation of $K_n$ and set $G_1 = (K_n)_f$. As in the proof of \cite{lu-2000}, we create a sequence of subtournaments of $G_1$ by deleting dominating sets greedily. Given $G_i$, let $x_i$ be a vertex with maximum out-degree in $G_i$. Let $V_{i+1} = N_{G_i}^-(x_i)$ and let $G_{i+1}$ be the subtournament induced by $V_{i+1}$. Continue this process until either ${r= \left\lceil \log_2 n - \log_2(\log_2 n) + 1\right\rceil}$ or $V_{r+1} = \emptyset$. Let $X=\{x_1,\ldots , x_r\}$ and let  $S= V(G_1) - X - V_{r+1}$.

Suppose first that $r \le \left\lceil \log_2 n - \log_2(\log_2 n) + 1\right\rceil$ and $V_{r+1} = \emptyset$. Since $f$ is valid  there exists $s \in S$ such that $(s, x_r) \in A(G_1)$. We claim that  $X \cup \{s\}$ is a total dominating set of $G_1$ of cardinality at most $\left\lceil \log_2 n - \log_2(\log_2 n) + 1\right\rceil + 1$. To see this, note that $X$ dominates $S$ and for each $i \in [r-1]$, $x_{i+1}$ dominates $x_i$. Thus, $X\cup \{s\}$ is a total dominating set of $G_1$ and we are done. 

Thus, we shall assume that the process terminates when $V_{r+1}\ne \emptyset$ and $r = \left\lceil \log_2 n - \log_2(\log_2 n) + 1\right\rceil$. Suppose first that some $s \in S$ dominates $V_{r+1}$ and let $v \in V_{r+1}$. Using a similar argument to that above, one can see that $T = X \cup \{s, v\}$ is a total dominating set of $G_1$ of cardinality $r+2 = \left\lceil \log_2 n - \log_2(\log_2 n) + 1\right\rceil + 2$ and we are done. 

Finally, suppose no vertex in $S$ dominates $V_{r+1}$. In this case, $V_{r+1}$ dominates $G_1$. Suppose first that $G[V_{r+1}]$ contains no vertex with in-degree $0$. Thus, $V_{r+1}$ is a total dominating set of $G_1$ of cardinality $|V_{r+1}|$. Therefore, we shall assume that $v \in V_{r+1}$ has in-degree $0$ in $G[V_{r+1}]$. Pick any in-neighbor $s \in S$ of $v$. Since $V_{r+1}$ dominates $s$, there exists $w \in V_{r+1}$ that dominates $s$. Now $V_{r+1} \cup \{s\}$ is a total dominating set of $G_1$ of cardinality $|V_{r+1}| + 1$. In either case, we have a total dominating set of $G_1$ of cardinality at most $|V_{r+1}| +1$. Note that $|V_{r+1}| \le n/2^r$ because in every tournament of order $p$ there exists a vertex with out-degree at least $p/2$. It follows that $G_1$ contains a total dominating set of cardinality at most 
\begin{eqnarray*}
|V_{r+1}| + 1 &\le& \frac{n}{2^r}+1\\
&\le& \frac{n}{2^{\log_2n - \log_2(\log_2 n)+1}} +1 \\&=& \frac{\log_2(n)}{2} +1\\
&<& \log_2n - 2\log_2(\log_2 n) + 2.
\end{eqnarray*} 
\end{proof}

We continue this section by obtaining exact values for the lower and upper orientable total domination numbers of complete bipartite graphs. The upper and lower orientable total domination numbers of complete multipartite graphs were already studied in~\cite{Lai-2001}. However, their results do not hold for complete bipartite graphs as the proofs start by choosing three vertices from three different partite sets.

\begin{prop}\label{CompleteBipartiteSmallDom}
For $2 \leq m \leq n$, $\dt(K_{m, n})  = 4$.
\end{prop}
\begin{proof}
Let $G = K_{m,n}$ and write $V(G) = X \cup Y$ where $X = \{x_1, \dots, x_m\}$ and $Y= \{y_1, \dots, y_n\}$ are the partite sets of $G$. By Observation \ref{tdsetcycle}, we know that for any orientation $G_h$ of $G$ every component of any total dominating set must contain at least one cycle. Since $G$ is triangle-free, it must be the case that $4 \leq \dt(G)$.

Let $G_f$ be an orientation of $G$ containing arcs $(x_1,y_1), (y_1,x_2), (x_2,y_2),$ and $(y_2,x_1)$ as well as $(x_1,y_j)$ for all $3 \leq j \leq n$ and $(y_1,x_i)$ for all $3 \leq i \leq m$. Note that this is a valid orientation since each vertex has in-degree at least one. In addition, $\{x_1, x_2, y_1, y_2\}$ is a total domination set. Hence, $\dt(G) \leq 4$. 

\end{proof}

\begin{thm} \label{CompleteBipartiteBigDom}
For $2 \leq m \leq n$, $\Dt(K_{m, n}) = m + 2$. 
\end{thm}
\begin{proof}
Let $G = K_{m,n}$ and write $V(G) = X \cup Y$ where $X = \{x_1, \dots, x_m\}$ and $Y= \{y_1, \dots, y_n\}$ are the partite sets of $G$. For the sake of contradiction, assume there exists an orienting mapping $h$ where $\gamma_t(G_h) \geq m + 3$, and let $A$ be a $\gamma_t$-set of $G_h$. Assume $|A \cap X| = k$ (where $k>0$ as $X$ is an independent set) and $|A \cap Y| \geq m - k + 3$.  Relabel the vertices if necessary so that $A \cap X = \{x_1, \ldots, x_k\}$. We know each $x_i$ for $i \in [k]$ has a private neighbor $y_i \in Y$ with respect to $A$, i.e.\ $y_i \in \textrm{pn}(x_i, A)$. Thus, $y_1$ dominates $\{x_2, \ldots, x_k\}$, and $y_2$ dominates $\{x_1, x_3,  \ldots, x_k\}$. Hence, $\{y_1, y_2\}$ dominates $\{x_1, \ldots, x_k\}$. Note that we can find a set $B \subseteq Y - \{y_1, y_2\}$ of cardinality at most $m-k$ such that $\{y_1, y_2 \} \cup B$ dominates $X$ as each vertex of $\{x_{k+1}, \dots, x_m\}$ has positive in-degree. Thus, $\{x_1, \ldots, x_k, y_1, y_2\}\cup B$ is a total dominating set of $G_h$ of cardinality at most $m+2$, which is a contradiction. Hence, for any orientation $G_h$ of $G$, $\gamma_t(G_h) \leq m + 2$ and $\Dt(G) \leq m + 2$.

Let $G_f$ be the orientation of $G$ with arcs $(y_i,x_i) \in A(G_f)$ and $(x_i,y_j) \in A(G_f)$ for all $i \in [m]$ and $j \in [n]$ such that $i \neq j$. Let $A$ be a $\gamma_t$-set of $G_f$. Since each $x_i$ has in-degree $1$, $\{y_1, \ldots, y_m\} \subseteq A$. In addition, to dominate $Y$ we need at least two vertices from $X$ since the out-degree of each $x_i$ is $n - 1$. Thus, $|A| \geq m + 2$. Since $\{x_1, x_2, y_1, \ldots, y_m\}$ is a total dominating set, $\gamma_t(G_f) = m + 2$. Hence, $\Dt(G) \geq m + 2$. 

\end{proof}

In the next result we prove that Conjecture~\ref{conj1} holds for complete bipartite graphs.
\begin{thm}\label{CompleteBipartiteConj}
Let $2 \leq m \leq n$ and let $k$ be any integer between $4$ and $m+2$. Then there exists an orienting mapping $f$ of $K_{m,n}$ such that $\gt((K_{m,n})_f)=k$.
\end{thm}

\begin{proof}
Let $G = K_{m,n}$ and write $V(G) = X \cup Y$ where $X = \{x_1, \dots, x_m\}$ and $Y= \{y_1, \dots, y_n\}$ are the partite sets of $G$.
If $k \in \{4,m+2\}$, the result follows from Theorem~\ref{CompleteBipartiteSmallDom} or Theorem~\ref{CompleteBipartiteBigDom}. Thus, let $5 \leq k \leq m+1$.  Let $G_f$ be the orientation of $G$ with arcs $(y_i,x_i) \in A(G_f)$ for all $i \in [k-3]$, $(y_{k-2}, x_i) \in A(G_f)$ for all $i \in \{k-2,k-1,\ldots , m\}$ and with all other edges directed from $X$ to $Y$. Let $A$ be a $\gamma_t$-set of $G_f$. Since each $x_i$ has in-degree $1$, $\{y_1, \ldots, y_{k-2}\} \subseteq A$. In addition, to dominate $Y$ we need at least two vertices from $X$ since the out-degree of each $x_i$ is $n - 1$. Thus, $|A| \geq k$. Since $\{x_1, x_2, y_1, \ldots, y_{k-2}\}$ is a total dominating set of $G_f$, $\gamma_t(G_f) = k$. 

\end{proof}

\section{Grid graphs}\label{sec:prod}

In this section we consider Cartesian products where both factors are paths. First, we give the exact value of $\dt(K_2\Box P_m)$ for $m \ge 3$. 
\vskip2mm
\begin{thm} For $m\ge 3$, \[\dt(K_2\Box P_m) =\begin{cases}  m& \text{ if  $m \equiv 0\pmod{4}$},\\
m+1 &\text{otherwise}.
\end{cases}
\] 
\end{thm}

\begin{proof} 
Label the vertices of $G = K_2 \Box P_m$ as $\{u_1, \dots, u_m, v_1, \dots, v_m\}$ such that $u_1,\dots, u_m$ induces a path, $v_1,\dots ,v_m$ induces a path, and $u_iv_i\in E(G)$ for $i \in [m]$. Let \[ p =\begin{cases}  m& \text{ if  $m \equiv 0\pmod{4}$},\\
m+1 &\text{otherwise}.
\end{cases}
\]  

First we can easily find orienting mapping $f$ of $G$ such that $\gt(G_f) \leq p$ and hence $\dt(G) \leq p$. Let $m=4k+o$ for some $o \in \{0,1,2,3\}$.  If  $o \neq 0$, let $A=\{u_i,u_{i+1},v_i,v_{i+1}: i =2+4\ell \text{ where } 0 \leq \ell < k \} \cup \{u_{4k},u_{4k+1},\ldots , u_{4k+o}\}$. If $o=0$, let $A=\{u_i,u_{i+1},v_i,v_{i+1}: i =2+4\ell \text{ where } 0 \leq \ell < k \}$. If we take any valid orientation $G_f$ in which 1) any edge $e=xy \in E(G)$ with one end vertex in $A$, say $x$, and the other end vertex in $V(G)-A$, is directed from $x$ to $y$, and 2) for fixed $0 \le \ell < k$ and $i = 2+4\ell$, $\{u_i, u_{i+1}, v_i, v_{i+1}\}$ induces an oriented cycle and if $o \ne 0$, $\{u_{4k-1}, \dots, u_{4k+o}\}$ induces a directed path, then $A$ is a total dominating set of $G_f$. Hence, $\dt(G) \leq \gt(G_f) \leq |A|=p$.

For the converse inequality, let $G_f$ be a valid orientation of $G$ and assume  that $A$ is a $\gt$-set of $G_f$. If $\{u_i, v_i\}\cap A = \emptyset$ for at most one $i \in [m]$, then $|A| \ge m+1$ as $G_f[A]$ must contain a cycle. Therefore, we let $\{j_1, \dots, j_t\}=\{i:\{u_i, v_i\} \cap A = \emptyset\}$ and we may assume $t\ge 2$. Without loss of generality we may assume $j_1 < j_2 < \ldots < j_t$. 

Suppose first that $j_1\ne1$. Note that this implies $j_1\ne 2$ for otherwise the component of $G_f[A]$ that contains $\{u_1,v_1\} \cap A$ does not contain a cycle. Since $\{u_i, v_i\}\cap A\ne \emptyset$ for any $i < j_1$ and as $G_f[\{u_i, v_i: i<j_1\}\cap A]$ must contain a cycle, it follows that $|\{u_i, v_i: i<j_1\}\cap A|\ge j_1-1+2 = j_1+1$. Moreover, $A\cap \{u_i, v_i: i>j_1\}$ is a total dominating set of $G_f[\{u_i, v_i: i > j_1\}]$ (which is a valid orientation since $u_{j_1},v_{j_1} \notin A$ and thus in $G_f$ $u_{j_1+1},v_{j_1+1}$ are dominated by a vertex in $\{u_i,v_i: i > j_1\}\cap A $). By the induction hypothesis, $|A\cap \{u_i, v_i: i>j_1\}|\ge m-j_1$ from which it follows that $|A| \ge m-j_1 + j_1+1$. Therefore, it remains to check the result for the case $j_1=1$. Since $j_1=1$, $A \cap \{u_1,v_1\} = \emptyset$ and thus $u_2,v_2 \in A$ as $A$ is a total dominating set of $G_f$. 

Assume first that $m \not\equiv 0\pmod{4}$. As above, $A \cap \{u_i, v_i\}\ne \emptyset$ when $2 \le i \le j_2-1$ and $G_f[A\cap\{u_i, v_i:2 \le i \le j_1-1\}]$ must contain a cycle so we have $|A\cap \{u_i, v_i:2 \le i \le j_2-1\}| \ge j_2-2+2 = j_2$. If $t\ge 3$, then we distinguish two cases. First, let $j_3 > j_2+1$ (i.e.\ $\{u_{j_2+1},v_{j_2+1}\} \cap A \neq \emptyset$). Then \[|A \cap \{u_i, v_i: j_2 < i < j_3\}| \ge j_3-1-j_2 + 2 = j_3-j_2+1\] by the same reasoning as above. If $j_3=m$, then it follows that $|A| \ge j_2 + j_3-j_2+1 = j_3+1 =m+1$. Otherwise, $m>j_3$ and $|A \cap \{u_i, v_i: i>j_3\}| \ge m-j_3$ by the induction hypothesis. In this case, $|A| \ge j_2 + j_3-j_2+1 + m-j_3 = m+1$. For the second case, let $j_3=j_2+1$, i.e. $A \cap \{u_{j_2},v_{j_2},u_{j_2+1},v_{j_2+1}\}= \emptyset$. Hence $A\cap \{u_i, v_i:2 \le i \le j_2-1\}$ is a total dominating set of $G_f[\{u_i, v_i:1 \le i \le j_2\}]$ and $A\cap \{u_i, v_i: i \geq j_3+1\}$ is a total dominating set of $G_f[\{u_i, v_i: i \geq  j_3\}]$. Since $m \not\equiv 0\pmod{4}$, either $j_2 \not\equiv 0\pmod{4}$ or $m-j_2 \not\equiv 0\pmod{4}$. Thus, by the induction hypothesis $|A\cap \{u_i, v_i:1 \le i \le j_2\}| > j_2$ or $|A\cap \{u_i, v_i: i \geq j_3\}| > m-j_2$ and we have  $|A| \geq j_2+m-j_2+1=m+1.$ Thus, we may assume $t=2$ and again we have $|A| \ge j_2 + m-j_2+2 = m+2$. 

Finally, assume that $m \equiv 0\pmod{4}$. As above, we may assume $|A\cap \{u_i, v_i:2 \le i \le j_2-1\}| \ge  j_2$ and by the induction hypothesis, $|A\cap \{u_i, v_i: j_2<i\le m\}|\ge m-j_2$ and we are done.

\end{proof}

We need the following lemma to present a lower bound for the lower orientable total domination number of grids.

\begin{lem}\label{l:neighborhood}
Let $3 \le m \le n$ and let $H$ be any connected subgraph of $G = P_m\Box P_n$ such that $H$ contains a cycle. Then $|N_G[V(H)]| \le 3|V(H)|$. 
\end{lem}

\begin{proof}
We proceed by induction on the number of leaves in $G[V(H)]$. Note that if $G[V(H)]$ contains no leaves, then each vertex in $H$ has at least two neighbors in $H$ and therefore at most two neighbors in $G-H$. Therefore, $|N_G[V(H)]| \le 3|V(H)|$. Now assume for some $\rho \in \mathbb{N}$ that if $G[V(H)]$ contains at most $\rho-1$ leaves, then $|N_G[V(H)]| \le 3|V(H)|$. Let $H$ be such that $G[V(H)]$ contains $\rho$ leaves and let $(a, b)$ be a leaf of $H$. Let $(u, v)$ be the vertex in $H$ where $\deg_H(u, v) \ge 3$ (note that such a vertex exists as  $H$ contains a cycle and a leaf) and $d_H((u,v)(a,b))$ is minimum. Denote $(a,b)=(x_k,y_k)$ and let 
\[(u, v)(x_1, y_1), \dots, (x_k, y_k)\] be a shortest $(u,v),(a,b)$-path in $H$. Let \[H' = H - \left(\bigcup_{i=1}^k\{(x_i, y_i)\}\right).\] Since $H'$ contains at most $\rho - 1$ leaves, induction hypothesis implies that $|N_G[V(H')]| \le 3|V(H')|$. Note that $(x_1, y_1) \in N_G[V(H')]$. For any $i \in [k-1]$ it holds that $(x_i, y_i)$ has at most two neighbors in $G-H$ and $(a, b)$ has at most $3$ neighbors in $G-H$. Hence, 
\begin{eqnarray*}
|N_G[V(H)]| &\le& |N_G[V(H')] - \{(x_1, y_1)\}| + 3(k-1)+4\\
&\le& 3|V(H')| -1+3k+1\\
&=& 3(|V(H')|+k)\\
&=& 3|V(H)|.
\end{eqnarray*}

\end{proof}

\begin{thm}\label{LowerBounddom}
If $3 \leq m \leq n$, then $\dt(P_m \Box P_n) \geq \frac{mn}{3}$. 
\end{thm}
\begin{proof}
Let $G=P_m \Box P_n$ and let $G_f$ be any valid orientation of $G$. Let $S$ be a $\gt$-set of $G_f$. Then it follows from Observation~\ref{tdsetcycle} that any  component of $G_f[S]$ contains a cycle. Thus, it follows from Lemma~\ref{l:neighborhood} that $S$ dominates at most $3|S|$ vertices of $G_f$. Since any total dominating set of $G_f$ must dominate all $|V(G_f)|$ vertices, it follows that  $\gt(G_f) \geq \frac{mn}{3}.$ Since this holds for any orientation $f$, we get $\dt(G) \geq \frac{mn}{3}$.
\end{proof}

The total domination number of (undirected) grids was studied in~\cite{Gravier-2002} in which  exact values of grids where one factor is small were obtained and upper and lower bounds were presented for general grids. For $16< m \le n$, the following bounds were given in~\cite{Gravier-2002} for $\gt(P_m\Box P_n)$.

\begin{thm}\cite{Gravier-2002}\label{thm:gtgrids} If $m$ and $n$ are integers greater than $16$, then  
\[\frac{3mn+2(m+n)}{12}-1 \le \gt(P_m\Box P_n) \le \left\lfloor \frac{(m+2)(n+2)}{4}\right\rfloor -4.\]
\end{thm}
Therefore, by Theorems~\ref{LowerBounddom} and \ref{thm:gtgrids}, $\dt(P_m \Box P_n) > \gt(P_m\Box P_n)$ for large $m$ and $n$.  In the next result we present an upper bound for $\dt(P_m \Box P_n)$. 

\begin{thm}\label{UpperBounddom}
If $3 \leq m \leq n$ and $3 \ | \  m$, then $\dt(P_m \Box P_n) \leq \frac{mn}{3}+\frac{2m}{3}$. 
\end{thm}

\begin{proof}
For $P_m = v_1 v_2 \dots v_m$ and $P_n = w_1 w_2 \dots w_n$, let $\{u_{i,j}:1\leq i \leq m, 1 \leq j \leq n \}$ be the vertices of $G=P_m \Box P_n$. \\

\noindent Consider the following orientation $f$ of $G$: 
(i) $(u_{i,j}, u_{i,j+1})$ for each $i \equiv 1 \pmod{3}$ where $1 \leq j \leq n-1$,\\
(ii) $(u_{i,j}, u_{i,j+1})$ for each $i \equiv 2\pmod{3}$ where $2 \leq j \leq n-1$,\\
(iii) $(u_{i,2}, u_{i,1})$ for each $i \equiv 2 \pmod{3}$,\\
(iv) $(u_{i,j}, u_{i,j+1})$ for each $i \equiv 0\pmod{3}$ where $1 \leq j \leq n-2$,\\
(v) $(u_{i,n}, u_{i,n-1})$ for each $i \equiv 0\pmod{3}$,\\
(vi) $(u_{i+1,j}, u_{i,j})$ for each $i \equiv 1 \pmod{3}$ where $1 \leq j \leq n$,\\
(vii) $(u_{i,j}, u_{i+1,j})$ for each $i \equiv 2 \pmod{3}$ where $j=1$ and $3 \leq j \leq n$,\\
(viii) $(u_{i+1,2}, u_{i,2})$ for each $i \equiv 2 \pmod{3}$,\\
(ix) $(u_{i,j}, u_{i+1,j})$ for each $i \equiv 0 \pmod{3}$ where $1 \leq j \leq n$.\\

This orientation is illustrated for $P_6 \Box P_8$ in Figure \ref{domP_6P_8}. Then let $S = \{ u_{i,j} \ :\  i \equiv 2 \pmod{3} \} \cup \{ u_{i,1}, u_{i,2}\  :\  i \equiv 0\pmod{3} \}$.  Note that $S$ will be $\frac{1}{3} m$ disjoint copies of a 4-cycle with a long path extending from one vertex.  We claim that $S$ is a $\gt$-set of $G_f$.  First, each $u_{i,j}$ with $i \equiv 2 \pmod{3}$ and $2 \leq j \leq n-1$ is the unique in-neighbor of the corresponding $u_{i,j+1}$  and must be in any $\gamma_t$-set of $G_f$.  Furthermore, for each $i \equiv 2 \pmod{3}$, $u_{i,n}$ is the unique in-neighbor of $u_{i+1,n}$ and must be in any $\gamma_t$-set of $G_f$ while $u_{i,2} , u_{i,1} , u_{i+1,1} , u_{i+1,2}$ forms a directed cycle of unique in-neighbors and must be in any $\gamma_t$-set of $G_f$.  So each vertex of $S$ is the unique in-neighbor of some vertex of $G_f$. That each vertex of $G_f$ is dominated by $S$, can be verified using the orientation given above.  Thus $S$ is a $\gamma_t$-set of $G_f$, and since $|S| = ( n + 2 ) ( \frac{1}{3} m ) =  \frac{mn}{3}+\frac{2m}{3}$, $\dt(P_m \Box P_n) \leq \frac{mn}{3}+\frac{2m}{3}$.

\end{proof}

\begin{figure}
\begin{center}
	\begin{tikzpicture}[]

	\draw [->, line width=1pt] (1,6) -- (1.8, 6);
	\draw [->, line width=1pt] (2,6) -- (2.8, 6);
	\draw [->, line width=1pt] (3,6) -- (3.8, 6);
	\draw [->, line width=1pt] (4,6) -- (4.8, 6);
	\draw [->, line width=1pt] (5,6) -- (5.8, 6);
	\draw [->, line width=1pt] (6,6) -- (6.8, 6);
	\draw [->, line width=1pt] (7,6) -- (7.8, 6);
	
	\draw [->, line width=1pt] (2,5) -- (1.2, 5);
	\draw [->, line width=1pt] (2,5) -- (2.8, 5);
	\draw [->, line width=1pt] (3,5) -- (3.8, 5);
	\draw [->, line width=1pt] (4,5) -- (4.8, 5);
	\draw [->, line width=1pt] (5,5) -- (5.8, 5);
	\draw [->, line width=1pt] (6,5) -- (6.8, 5);
	\draw [->, line width=1pt] (7,5) -- (7.8, 5);	
	
	\draw [->, line width=1pt] (1,4) -- (1.8, 4);
	\draw [->, line width=1pt] (2,4) -- (2.8, 4);
	\draw [->, line width=1pt] (3,4) -- (3.8, 4);
	\draw [->, line width=1pt] (4,4) -- (4.8, 4);
	\draw [->, line width=1pt] (5,4) -- (5.8, 4);
	\draw [->, line width=1pt] (6,4) -- (6.8, 4);
	\draw [->, line width=1pt] (8,4) -- (7.2, 4);
	
	\draw [->, line width=1pt] (1,3) -- (1.8, 3);
	\draw [->, line width=1pt] (2,3) -- (2.8, 3);
	\draw [->, line width=1pt] (3,3) -- (3.8, 3);
	\draw [->, line width=1pt] (4,3) -- (4.8, 3);
	\draw [->, line width=1pt] (5,3) -- (5.8, 3);
	\draw [->, line width=1pt] (6,3) -- (6.8, 3);
	\draw [->, line width=1pt] (7,3) -- (7.8, 3);
	
	\draw [->, line width=1pt] (2,2) -- (1.2, 2);
	\draw [->, line width=1pt] (2,2) -- (2.8, 2);
	\draw [->, line width=1pt] (3,2) -- (3.8, 2);
	\draw [->, line width=1pt] (4,2) -- (4.8, 2);
	\draw [->, line width=1pt] (5,2) -- (5.8, 2);
	\draw [->, line width=1pt] (6,2) -- (6.8, 2);
	\draw [->, line width=1pt] (7,2) -- (7.8, 2);
	
	\draw [->, line width=1pt] (1,1) -- (1.8, 1);
	\draw [->, line width=1pt] (2,1) -- (2.8, 1);
	\draw [->, line width=1pt] (3,1) -- (3.8, 1);
	\draw [->, line width=1pt] (4,1) -- (4.8, 1);
	\draw [->, line width=1pt] (5,1) -- (5.8, 1);
	\draw [->, line width=1pt] (6,1) -- (6.8, 1);
	\draw [->, line width=1pt] (8,1) -- (7.2, 1);
	
	
	\draw [->, line width=1pt] (1,5) -- (1, 5.8);
	\draw [->, line width=1pt] (2,5) -- (2, 5.8);
	\draw [->, line width=1pt] (3,5) -- (3, 5.8);
	\draw [->, line width=1pt] (4,5) -- (4, 5.8);
	\draw [->, line width=1pt] (5,5) -- (5, 5.8);
	\draw [->, line width=1pt] (6,5) -- (6, 5.8);
	\draw [->, line width=1pt] (7,5) -- (7, 5.8);
	\draw [->, line width=1pt] (8,5) -- (8, 5.8);
	
	\draw [->, line width=1pt] (1,5) -- (1, 4.2);
	\draw [->, line width=1pt] (2,4) -- (2, 4.8);
	\draw [->, line width=1pt] (3,5) -- (3, 4.2);
	\draw [->, line width=1pt] (4,5) -- (4, 4.2);
	\draw [->, line width=1pt] (5,5) -- (5, 4.2);
	\draw [->, line width=1pt] (6,5) -- (6, 4.2);
	\draw [->, line width=1pt] (7,5) -- (7, 4.2);
	\draw [->, line width=1pt] (8,5) -- (8, 4.2);
	
	\draw [->, line width=1pt] (1,4) -- (1, 3.2);
	\draw [->, line width=1pt] (2,4) -- (2, 3.2);
	\draw [->, line width=1pt] (3,4) -- (3, 3.2);
	\draw [->, line width=1pt] (4,4) -- (4, 3.2);
	\draw [->, line width=1pt] (5,4) -- (5, 3.2);
	\draw [->, line width=1pt] (6,4) -- (6, 3.2);
	\draw [->, line width=1pt] (7,4) -- (7, 3.2);
	\draw [->, line width=1pt] (8,4) -- (8, 3.2);

	\draw [->, line width=1pt] (1,2) -- (1, 2.8);
	\draw [->, line width=1pt] (2,2) -- (2, 2.8);
	\draw [->, line width=1pt] (3,2) -- (3, 2.8);
	\draw [->, line width=1pt] (4,2) -- (4, 2.8);
	\draw [->, line width=1pt] (5,2) -- (5, 2.8);
	\draw [->, line width=1pt] (6,2) -- (6, 2.8);
	\draw [->, line width=1pt] (7,2) -- (7, 2.8);
	\draw [->, line width=1pt] (8,2) -- (8, 2.8);
	
	\draw [->, line width=1pt] (1,2) -- (1, 1.2);
	\draw [->, line width=1pt] (2,1) -- (2, 1.8);
	\draw [->, line width=1pt] (3,2) -- (3, 1.2);
	\draw [->, line width=1pt] (4,2) -- (4, 1.2);
	\draw [->, line width=1pt] (5,2) -- (5, 1.2);
	\draw [->, line width=1pt] (6,2) -- (6, 1.2);
	\draw [->, line width=1pt] (7,2) -- (7, 1.2);
	\draw [->, line width=1pt] (8,2) -- (8, 1.2);


	\draw [fill=gray] (1, 1) circle (5pt);
	\draw [fill=gray] (2, 1) circle (5pt);
	\draw [fill=white] (3, 1) circle (5pt);
	\draw [fill=white] (4, 1) circle (5pt);
	\draw [fill=white] (5, 1) circle (5pt);
	\draw [fill=white] (6, 1) circle (5pt);
	\draw [fill=white] (7, 1) circle (5pt);
	\draw [fill=white] (8, 1) circle (5pt);

	\draw [fill=gray] (1, 2) circle (5pt);
	\draw [fill=gray] (2, 2) circle (5pt);
	\draw [fill=gray] (3, 2) circle (5pt);
	\draw [fill=gray] (4, 2) circle (5pt);
	\draw [fill=gray] (5, 2) circle (5pt);	
	\draw [fill=gray] (6, 2) circle (5pt);
	\draw [fill=gray] (7, 2) circle (5pt);
	\draw [fill=gray] (8, 2) circle (5pt);
	
	\draw [fill=white] (1, 3) circle (5pt);
	\draw [fill=white] (2, 3) circle (5pt);
	\draw [fill=white] (3, 3) circle (5pt);
	\draw [fill=white] (4, 3) circle (5pt);
	\draw [fill=white] (5, 3) circle (5pt);	
	\draw [fill=white] (6, 3) circle (5pt);
	\draw [fill=white] (7, 3) circle (5pt);
	\draw [fill=white] (8, 3) circle (5pt);
	
	\draw [fill=gray] (1, 4) circle (5pt);
	\draw [fill=gray] (2, 4) circle (5pt);
	\draw [fill=white] (3, 4) circle (5pt);
	\draw [fill=white] (4, 4) circle (5pt);
	\draw [fill=white] (5, 4) circle (5pt);	
	\draw [fill=white] (6, 4) circle (5pt);
	\draw [fill=white] (7, 4) circle (5pt);
	\draw [fill=white] (8, 4) circle (5pt);
	
	\draw [fill=gray] (1, 5) circle (5pt);
	\draw [fill=gray] (2, 5) circle (5pt);
	\draw [fill=gray] (3, 5) circle (5pt);
	\draw [fill=gray] (4, 5) circle (5pt);
	\draw [fill=gray] (5, 5) circle (5pt);	
	\draw [fill=gray] (6, 5) circle (5pt);
	\draw [fill=gray] (7, 5) circle (5pt);
	\draw [fill=gray] (8, 5) circle (5pt);
	
	\draw [fill=white] (1, 6) circle (5pt);
	\draw [fill=white] (2, 6) circle (5pt);
	\draw [fill=white] (3, 6) circle (5pt);
	\draw [fill=white] (4, 6) circle (5pt);
	\draw [fill=white] (5, 6) circle (5pt);	
	\draw [fill=white] (6, 6) circle (5pt);
	\draw [fill=white] (7, 6) circle (5pt);
	\draw [fill=white] (8, 6) circle (5pt);
	
	\draw (0.2,6) node{$u_{1,1}$}; 
	\draw (0.2,1) node{$u_{6,1}$}; 
	\draw (8.8,6) node{$u_{1,8}$}; 
	\draw (8.8,1) node{$u_{6,8}$}; 
	
	\end{tikzpicture} \\
	\caption{Orientation of $P_6 \Box P_8$ with  $\gt$-set $S$ indicated in gray.}\label{domP_6P_8}
	
\end{center}
\end{figure}

\begin{cor}\label{UpperBounddomCor1}
If $3 \leq m \leq n$ and $m \equiv 2\pmod{3} $, then $\dt(P_m \Box P_n) \leq \frac{(m+1)n}{3}+\frac{2(m+1)}{3}$. 
\end{cor}

\begin{proof}
This result follows from removing the first row $\{u_{1,1}, u_{1,2}, \dots , u_{1,n} \}$ from the above orientation.
\end{proof}

\begin{cor}\label{UpperBounddomCor2}
If $3 \leq m \leq n$ and $m \equiv 1 \pmod{3} $, then $\dt(P_m \Box P_n) \leq \frac{(m+2)n}{3}+\frac{2(m+2)}{3}$. 
\end{cor}

\begin{proof}
This result follows from removing the paths corresponding to $\{u_{1,1}, u_{1,2}, \dots $ $, u_{1,n} \}$ and $\{u_{m,1}, u_{m,2}, \dots , u_{m,n} \}$ from the above orientation, reorienting the edge between $u_{m-1, 2}$ and $u_{m-2,2}$, and adding vertices $u_{m-2, 1}$ and $u_{m-2,2}$ to the set $S$.  
\end{proof}

Next we will provide upper and lower bounds on the upper orientable total domination number of the ladder, $K_2 \Box P_m$. In order to do so, we first consider $K_2\Box P_3$.

\begin{lem}\label{P3xP2} 
Let $f$ be any valid orientation of a graph $G=K_2\Box P_3$. Then there exists a vertex in $G_f$ with out-degree at least 2.  
\end{lem}
\begin{proof}
Let $G=K_2\Box P_3$ and write  $V(G)=\{u_1,u_2,u_3,v_1,v_2,v_3\}$, where $\{u_1,u_2,u_3\}$ induces a path, $\{v_1,v_2,v_3\}$ induces a path and $u_iv_j \in E(G)$ if and only if $i=j$. Let $f$ be any valid orientation of $G$ and assume that each vertex of $G_f$ has out-degree at most 1. Without loss of generality let $(u_1,v_1) \in A(G_f)$. Since $u_1$ has out-degree at most one, $(u_2,u_1) \in A(G_f)$ and thus $(v_2,u_2),(u_3,u_2) \in A(G_f)$. Again since also $v_2$ and $u_3$ have out-degree at most 1, $(v_3,v_2), (v_3,u_3) \in A(G_f)$. Hence $v_3$ has two out-neighbors $v_2$ and $u_3$, a contradiction.
\end{proof}

\begin{thm}\label{UpperBoundDOM}
Let $3 \leq m$ and let $G=K_2 \Box P_m$. Then $\Dt(K_2 \Box P_m) \leq  \lceil \frac{5}{6}n(G) \rceil$. 

\end{thm}
\begin{proof}
Let $G=K_2 \Box P_m$, $m=3k+o$ and let $f$ be a valid orientation of $G$ such that $\Dt(G)=\gt(G_f)$. Write $V(G)=\{u_1,\ldots ,u_m,v_1,\ldots , v_m\}$, where $\{u_1,\ldots , u_m\}$ induces a path, $\{v_1,\ldots , v_m\}$ induces a path and $u_iv_j \in E(G)$ if and only if $i=j$. Partition $V(G)$ into sets $A_i=\{u_{3i+1}, u_{3i+2},u_{3i+3},v_{3i+1},v_{3i+2},v_{3i+3}\}$ for $i \in \{0,\ldots , k-1\}$ and $A_k=V(G) \setminus (A_0 \cup \cdots \cup A_{k-1})$. Moreover, for $i \in \{0,1,\ldots , k\}$ let $(H_i)_f$ be the subgraph of $G_f$ induced by $A_i$. If the orientation $(H_i)_f$ is not valid, then there is a vertex $v$ in $(H_i)_f$ with in-degree 0. Since $\delta(G)=2$, $v$ has out-degree at least 2. Thus it follows from Lemma~\ref{P3xP2} that $(H_i)_f$ has a vertex $v_i$ with out-degree at least 2 for any $i \in \{0,1,\ldots , k-1\}$. Now, we construct a total dominating set $S$ of $G_f$ in the following way. Let $\{v_0,v_1,\ldots , v_{k-1}\} \subseteq S$. These vertices dominate at least $2k$ vertices in $G_f$. For each of the remaining (possibly undominated) $2m-2k=4k+2o$ vertices, choose an arbitrary in-neighbor in $G_f$ and add it to $S$. Thus $S$ is a total dominating set of cardinality $|S|=k+4k+2o=5k+2o=\lceil \frac{5}{6}(6k+2o) \rceil$. Hence $\Dt(G)=\gt(G_f)\leq |S|\leq \lceil \frac{5}{6}n(G) \rceil$.  
 
\end{proof}

\begin{thm}\label{LowerBoundLadder}
If $3 \leq m$ and $G=K_2 \Box P_m$, then $\Dt(K_2 \Box P_m) \geq  \lfloor \frac{3}{2}m \rfloor +1$. 
\end{thm}
\begin{proof}
Label the vertices of $G=K_2 \Box P_m$ as $\{u_1, \dots, u_m, v_1, \dots, v_m\}$ such that $\{u_1, \cdots , u_m\}$ induces a path, $\{v_1, \cdots , v_m\}$ induces a path, and $u_iv_i\in E(G)$ for $i \in [m]$.

Let $f$ be the following orienting mapping of $G$. Direct the outer cycle of the grid cyclically, i.e.\ $(u_i,u_{i+1}),(v_{i+1},v_i) \in A(G_f)$ for any $i \in [m-1]$, $(u_m,v_m),(v_1,u_1) \in A(G_f)$. For any $i \in \{2,\ldots , m-1\}$ direct the edge $u_iv_i \in E(G)$ as $(u_i,v_i)$ if $i$ is even and as $(v_i,u_i)$ if $i$ is odd. Let $S$ be $\gt$-set of $G_f$. Since for any $\ell$ $u_{2\ell}$ and $v_{2\ell -1}$ both have in-degree 1, its in-neighbors $u_{2\ell -1}, v_{2\ell}$ must be in $S$. Moreover, $v_1,u_m \in S$, since $u_1$ and $v_m$ also have in-degree 1. For any $k \in \{1,2,\ldots , \lfloor \frac{m-2}{2} \rfloor\}$ $v_{2k}$ has exactly two in-neighbors $u_{2k}, v_{2k+1}$. Since $N^+(u_{2k})=N^+(v_{2k+1})=\{v_{2k},u_{2k+1}\}$ exactly one vertex from $\{ u_{2k},v_{2k+1}\}$ is in $S$. If $m$ is odd, then also $v_{m-1}$ has exactly two out-neighbors $u_{m-1}$ and $v_m$. Since $N^+(v_m) \subset N^+(u_{m-1})$, again $S$ cannot contain both $v_m$ and $u_{m-1}$, but it must contain one of them. Thus $|S| \geq \lfloor \frac{3}{2}m \rfloor +1.$ 


\end{proof}

Finally, we give a lower bound on the upper orientable total domination number of $P_m \Box P_n$.

\begin{thm}\label{UpperBoundDOM}
If $3 \leq m \leq n$, then $\Dt(P_m \Box P_n) \geq  \frac{1}{2}mn  + \frac{1}{2}m$. 
\end{thm}

\begin{proof}
For $P_m = v_1 v_2 \dots v_m$ and $P_n = w_1 w_2 \dots w_n$, let $\{u_{i,j}:1\leq i \leq m, 1 \leq j \leq n \}$ be the vertices of $G=P_m \Box P_n$. Consider the following orientation $f$ of $G$: \\
(i) $(u_{1,2}, u_{1,1})$ and $(u_{i,j}, u_{i,j+1})$ when $i$ and $j$ are not both equal to $1$,\\
(ii) $(u_{2,2}, u_{1,2})$ and $(u_{1,j},u_{2,j})$ for $j=1$ and $3 \leq j \leq n$,\\
(iii) $(u_{i+1, j}, u_{i, j})$ for each even $i \geq 2$ and $2 \leq j \leq n - 1$, \\
(iv) $(u_{i,j}, u_{i+1,j})$ for each odd $i \geq 3$ and $2 \leq j \leq n - 1$, \\
(v) $(u_{i,1}, u_{i+1,1})$ and $(u_{i,n},u_{i+1,n})$ for each $i \geq 2$.\\

This orientation is illustrated for $P_6 \Box P_8$ in Figure \ref{DomP_6P_8}. For even $m \geq 3$, let $S = \{ u_{i,j} : i \,\, \text{is odd}, 1 \leq j \leq n \} \cup \{u_{i,1} : i \,\, \text{is even}, i \neq m\} \cup \{u_{2,2}\}$.  We claim that $S$ is a minimum $\gt$-set of $G_f$.   Since $u_{1, 1}, u_{1, 2}, u_{2, 1}, u_{2, 2}$ are the unique in-neighbors of $u_{2, 1}, u_{1, 1}, u_{3, 1}$, and $u_{1, 2}$, respectively, they must be in any $\gamma_t$-set of $G_f$. Next, for $3 \leq j \leq n-1$, each $u_{1,j}$ is the unique in-neighbor of $u_{1,j+1}$ and must be in any $\gamma_t$-set of $G_f$. Similarly, for odd $i \geq 3$ and $2 \leq j \leq n-2$, each $u_{i,j}$ is the unique in-neighbor of $u_{i,j+1}$ and must be in any $\gamma_t$-set of $G_f$. Then, for $3 \leq i \leq m-1$, each $u_{i,1}$ is the unique in-neighbor of $u_{i+1,1}$ and must be in any $\gamma_t$-set of $G_f$.  Finally, for each $i\in \{2,\ldots , m\}$, either $u_{i,n-1}$ or $u_{i-1,n}$ must be in any $\gamma_t$-set of $G_f$ to dominate the vertex $u_{i,n}$.  We choose to include $u_{i-1,n} \in S$ for even $i$ and $u_{i, n - 1} \in S$ for odd $i$.  That each vertex of $G_f$ is dominated by $S$, can be verified using the orientation given above.  Since  $S$ is $\gamma_t$-set and $|S| = (m/2)(n) + (m-2)/2 + 1 = \frac{1}{2} mn + \frac{1}{2}m$, we have that $\Dt(P_m \Box P_n) \geq  \frac{1}{2}mn + \frac{1}{2}m$ for even $m \geq 3$.

When $m \geq 3$ is odd, we remove $u_{m,n}$ from $S$ to obtain a $\gamma_t$-set of size $|S| - 1 = \lceil m/2 \rceil (n) + \lceil (m-2)/2 \rceil = \frac{m + 1}{2} n + \frac{m - 1}{2} = \frac{1}{2} mn + \frac{1}{2}m + \frac{1}{2}n - \frac{1}{2} \geq \frac{1}{2} mn + \frac{1}{2} m$ since $n \geq 3$. This gives $\Dt(P_m \Box P_n) \geq  \frac{1}{2}mn + \frac{1}{2} m$ for odd $m \geq 3$.

\end{proof}

\begin{figure}
\begin{center}
	\begin{tikzpicture}[]
	\draw [->, line width=1pt] (2,6) -- (1.2, 6);
	\draw [->, line width=1pt] (2,6) -- (2.8, 6);
	\draw [->, line width=1pt] (3,6) -- (3.8, 6);
	\draw [->, line width=1pt] (4,6) -- (4.8, 6);
	\draw [->, line width=1pt] (5,6) -- (5.8, 6);
	\draw [->, line width=1pt] (6,6) -- (6.8, 6);
	\draw [->, line width=1pt] (7,6) -- (7.8, 6);
	
	\draw [->, line width=1pt] (1,5) -- (1.8, 5);
	\draw [->, line width=1pt] (2,5) -- (2.8, 5);
	\draw [->, line width=1pt] (3,5) -- (3.8, 5);
	\draw [->, line width=1pt] (4,5) -- (4.8, 5);
	\draw [->, line width=1pt] (5,5) -- (5.8, 5);
	\draw [->, line width=1pt] (6,5) -- (6.8, 5);
	\draw [->, line width=1pt] (7,5) -- (7.8, 5);	
	
	\draw [->, line width=1pt] (1,4) -- (1.8, 4);
	\draw [->, line width=1pt] (2,4) -- (2.8, 4);
	\draw [->, line width=1pt] (3,4) -- (3.8, 4);
	\draw [->, line width=1pt] (4,4) -- (4.8, 4);
	\draw [->, line width=1pt] (5,4) -- (5.8, 4);
	\draw [->, line width=1pt] (6,4) -- (6.8, 4);
	\draw [->, line width=1pt] (7,4) -- (7.8, 4);
	
	\draw [->, line width=1pt] (1,3) -- (1.8, 3);
	\draw [->, line width=1pt] (2,3) -- (2.8, 3);
	\draw [->, line width=1pt] (3,3) -- (3.8, 3);
	\draw [->, line width=1pt] (4,3) -- (4.8, 3);
	\draw [->, line width=1pt] (5,3) -- (5.8, 3);
	\draw [->, line width=1pt] (6,3) -- (6.8, 3);
	\draw [->, line width=1pt] (7,3) -- (7.8, 3);
	
	\draw [->, line width=1pt] (1,2) -- (1.8, 2);
	\draw [->, line width=1pt] (2,2) -- (2.8, 2);
	\draw [->, line width=1pt] (3,2) -- (3.8, 2);
	\draw [->, line width=1pt] (4,2) -- (4.8, 2);
	\draw [->, line width=1pt] (5,2) -- (5.8, 2);
	\draw [->, line width=1pt] (6,2) -- (6.8, 2);
	\draw [->, line width=1pt] (7,2) -- (7.8, 2);
	
	\draw [->, line width=1pt] (1,1) -- (1.8, 1);
	\draw [->, line width=1pt] (2,1) -- (2.8, 1);
	\draw [->, line width=1pt] (3,1) -- (3.8, 1);
	\draw [->, line width=1pt] (4,1) -- (4.8, 1);
	\draw [->, line width=1pt] (5,1) -- (5.8, 1);
	\draw [->, line width=1pt] (6,1) -- (6.8, 1);
	\draw [->, line width=1pt] (7,1) -- (7.8, 1);
	
	
	\draw [->, line width=1pt] (1,6) -- (1, 5.2);
	\draw [->, line width=1pt] (2,5) -- (2, 5.8);
	\draw [->, line width=1pt] (3,6) -- (3, 5.2);
	\draw [->, line width=1pt] (4,6) -- (4, 5.2);
	\draw [->, line width=1pt] (5,6) -- (5, 5.2);
	\draw [->, line width=1pt] (6,6) -- (6, 5.2);
	\draw [->, line width=1pt] (7,6) -- (7, 5.2);
	\draw [->, line width=1pt] (8,6) -- (8, 5.2);
	
	\draw [->, line width=1pt] (1,5) -- (1, 4.2);
	\draw [->, line width=1pt] (2,4) -- (2, 4.8);
	\draw [->, line width=1pt] (3,4) -- (3, 4.8);
	\draw [->, line width=1pt] (4,4) -- (4, 4.8);
	\draw [->, line width=1pt] (5,4) -- (5, 4.8);
	\draw [->, line width=1pt] (6,4) -- (6, 4.8);
	\draw [->, line width=1pt] (7,4) -- (7, 4.8);
	\draw [->, line width=1pt] (8,5) -- (8, 4.2);
	
	\draw [->, line width=1pt] (1,4) -- (1, 3.2);
	\draw [->, line width=1pt] (2,4) -- (2, 3.2);
	\draw [->, line width=1pt] (3,4) -- (3, 3.2);
	\draw [->, line width=1pt] (4,4) -- (4, 3.2);
	\draw [->, line width=1pt] (5,4) -- (5, 3.2);
	\draw [->, line width=1pt] (6,4) -- (6, 3.2);
	\draw [->, line width=1pt] (7,4) -- (7, 3.2);
	\draw [->, line width=1pt] (8,4) -- (8, 3.2);

	\draw [->, line width=1pt] (1,3) -- (1, 2.2);
	\draw [->, line width=1pt] (2,2) -- (2, 2.8);
	\draw [->, line width=1pt] (3,2) -- (3, 2.8);
	\draw [->, line width=1pt] (4,2) -- (4, 2.8);
	\draw [->, line width=1pt] (5,2) -- (5, 2.8);
	\draw [->, line width=1pt] (6,2) -- (6, 2.8);
	\draw [->, line width=1pt] (7,2) -- (7, 2.8);
	\draw [->, line width=1pt] (8,3) -- (8, 2.2);
	
	\draw [->, line width=1pt] (1,2) -- (1, 1.2);
	\draw [->, line width=1pt] (2,2) -- (2, 1.2);
	\draw [->, line width=1pt] (3,2) -- (3, 1.2);
	\draw [->, line width=1pt] (4,2) -- (4, 1.2);
	\draw [->, line width=1pt] (5,2) -- (5, 1.2);
	\draw [->, line width=1pt] (6,2) -- (6, 1.2);
	\draw [->, line width=1pt] (7,2) -- (7, 1.2);
	\draw [->, line width=1pt] (8,2) -- (8, 1.2);


	\draw [fill=white] (1, 1) circle (5pt);
	\draw [fill=white] (2, 1) circle (5pt);
	\draw [fill=white] (3, 1) circle (5pt);
	\draw [fill=white] (4, 1) circle (5pt);
	\draw [fill=white] (5, 1) circle (5pt);
	\draw [fill=white] (6, 1) circle (5pt);
	\draw [fill=white] (7, 1) circle (5pt);
	\draw [fill=white] (8, 1) circle (5pt);

	\draw [fill=gray] (1, 2) circle (5pt);
	\draw [fill=gray] (2, 2) circle (5pt);
	\draw [fill=gray] (3, 2) circle (5pt);
	\draw [fill=gray] (4, 2) circle (5pt);
	\draw [fill=gray] (5, 2) circle (5pt);	
	\draw [fill=gray] (6, 2) circle (5pt);
	\draw [fill=gray] (7, 2) circle (5pt);
	\draw [fill=gray] (8, 2) circle (5pt);
	
	\draw [fill=gray] (1, 3) circle (5pt);
	\draw [fill=white] (2, 3) circle (5pt);
	\draw [fill=white] (3, 3) circle (5pt);
	\draw [fill=white] (4, 3) circle (5pt);
	\draw [fill=white] (5, 3) circle (5pt);	
	\draw [fill=white] (6, 3) circle (5pt);
	\draw [fill=white] (7, 3) circle (5pt);
	\draw [fill=white] (8, 3) circle (5pt);
	
	\draw [fill=gray] (1, 4) circle (5pt);
	\draw [fill=gray] (2, 4) circle (5pt);
	\draw [fill=gray] (3, 4) circle (5pt);
	\draw [fill=gray] (4, 4) circle (5pt);
	\draw [fill=gray] (5, 4) circle (5pt);	
	\draw [fill=gray] (6, 4) circle (5pt);
	\draw [fill=gray] (7, 4) circle (5pt);
	\draw [fill=gray] (8, 4) circle (5pt);
	
	\draw [fill=gray] (1, 5) circle (5pt);
	\draw [fill=gray] (2, 5) circle (5pt);
	\draw [fill=white] (3, 5) circle (5pt);
	\draw [fill=white] (4, 5) circle (5pt);
	\draw [fill=white] (5, 5) circle (5pt);	
	\draw [fill=white] (6, 5) circle (5pt);
	\draw [fill=white] (7, 5) circle (5pt);
	\draw [fill=white] (8, 5) circle (5pt);
	
	\draw [fill=gray] (1, 6) circle (5pt);
	\draw [fill=gray] (2, 6) circle (5pt);
	\draw [fill=gray] (3, 6) circle (5pt);
	\draw [fill=gray] (4, 6) circle (5pt);
	\draw [fill=gray] (5, 6) circle (5pt);	
	\draw [fill=gray] (6, 6) circle (5pt);
	\draw [fill=gray] (7, 6) circle (5pt);
	\draw [fill=gray] (8, 6) circle (5pt);
	
	\draw (0.2,6) node{$u_{1,1}$}; 
	\draw (0.2,1) node{$u_{6,1}$}; 
	\draw (8.8,6) node{$u_{1,8}$}; 
	\draw (8.8,1) node{$u_{6,8}$}; 
	
	\end{tikzpicture} \\
	\caption{Orientation of $P_6 \Box P_8$ with  $\gt$-set $S$ indicated in gray.}\label{DomP_6P_8}
	
\end{center}
\end{figure}

Since cycles are the only graphs with both lower and upper orientable total domination numbers equal to their order, Vizing's conjecture does not hold for these two invariants. Indeed $\dt(C_3)=\Dt(C_3) = 3$ and $\dt(C_3 \Box C_3) \leq \Dt(C_3\Box C_3) < 9 = \dt(C_3)\dt(C_3) = \Dt(C_3)\Dt(C_3)$.

\section{Graphs with extreme orientable total domination number}\label{sec:extreme}

We already know that $\dt(G) = \Dt(G)=n(G)$ if and only if $G$ is a cycle.  In this section, we give a family of graphs $\mathcal{F}$ such that $\dt(G)=n(G)-1$ if and only if $G \in {\mathcal{F}} \cup \{K_4,K_{2,3},K_4-e\}$.

We first prove some necessary conditions for graphs with $\dt(G)=n(G)-1$. Note that this implies that for any valid orientation $f$, $\gt(G_f)=n(G)-1$. 
\begin{lem}\label{l:maxDegree}
If $G$ is a graph with $\dt(G)=n(G)-1$, then $\Delta(G) \leq 3$.
\end{lem}
\begin{proof}
Suppose that there is a vertex $x \in V(G)$ with deg$(x) \geq 4$. Let $f$ be a valid orienting mapping of $G$. Since $f$ is valid, there exists $y \in N_{G_f}^-(x)$. Let $G_h$ be an orientation obtained from $G_f$, such that all edges incident with $x$ in $G$, except the edge $xy$, are oriented out of $x$, but all other edges have the same direction as in $f$. Then $h$ is also a valid orientation of $G$ (note that $|N_{G_h}^-(v)| \geq |N_{G_f}^-(v)|$ holds for any $v \in V(G)\setminus \{x\}$). We can create a total dominating set $S$ in $G_h$ by adding $x$ to $S$ and choosing an in-neighbor for each of the remaining vertices not dominated by $x$. Since $x$ dominates at least three vertices of $G_h$, $\dt(G) \leq \gt(G_h) \leq n(G)-2$, a contradiction.
\end{proof}

\begin{lem}\label{l:degree3vertices}
If $G$ is a connected graph with $\dt(G)=n(G)-1$, then $G$ is isomorphic to $K_{2,3}$ or the vertices of degree 3 induce a complete graph. 
\end{lem}
\begin{proof}
Note first that $\dt(K_{2,3})=4=n(K_{2,3})-1$ by Theorem~\ref{CompleteBipartiteSmallDom}. Let $G$ be a graph not isomorphic to $K_{2,3}$ with $\dt(G)=n(G)-1$ and let $f$ be an arbitrary valid orienting mapping of $G$. Then $\Delta(G)\leq 3$ by Lemma~\ref{l:maxDegree}. For the purpose of contradiction assume that $G$ contains two non-adjacent vertices $x,y$ of degree 3. 

First let $N(x) \cap N(y)= \emptyset$. Since $f$ is valid, $x$ has at least one in-neighbor $x_1$ and $y$ has an in-neighbor $y_1$ in $G_f$. Then an orientation $h$ of $G$ obtained from $f$  by changing (if necessary) the direction of edges incident  with $x$ and $y$ such that $x_1$ will be the only in-neighbor of $x$ and $y_1$ will be the only in-neighbor of $y$ in $G_h$ is also a valid orientation, as the in-degree of any vertex from $V(G) \setminus \{x,y\}$ in $G_h$ is at least as large as in $G_f$. We can create a total dominating set $S$ of $G_h$ by including $x$ and $y$ in $S$ and then choosing one in-neighbor for each vertex not dominated by $x$ or $y$. Since $\{x, y\}$ dominates four vertices in $G_h$, $\gt(G_h) \leq n(G)-2$, which is a contradiction.

Now let $\{z\}=N(x) \cap N(y)$. If $z$ is an in-neighbor of at least one vertex $x$ or $y$, say $x$, then denote by $y'$ an arbitrary in-neighbor of $y$ in $G_f$ (note that $y'=z$ is also possible). Let
$h$ be an orienting mapping obtained from $f$ by changing (if necessary) the direction of edges incident  with $x$ and $y$ such that $z$ will be the only in-neighbor of $x$ and $y'$ will be the only in-neighbor of $y$ in $G_h$. Then $h$ is also a valid orienting mapping. Since $\{x, y\}$ dominates four vertices in $G_h$, $\gt(G_h) \leq n(G)-2$, a contradiction. If $z$ is an out-neighbor of both $x$ and $y$ in $G_f$, then we can change the direction of the edge $(x,z)$ to $(z,x)$ and leave all other edges directed as in $f$ and obtain another valid orientation of $G$, which leads us to the previous case, where we had an orientation in which $z$ is an in-neighbor of at least one vertex $x$ or $y$.

Now let $\{z_1,z_2\}=N(x) \cap N(y)$. Let $x_1 \in N(x) \setminus N(y)$ and $y_1 \in N(y) \setminus N(x)$. Let $h$ be the orienting mapping of $G$ obtained from $f$ by changing (if necessary) the direction of the edges incident with $x$ and $y$ such that $(x,x_1),(x,z_1),(y,y_1), (y,z_2),(z_1,y),$ $(z_2,x) \in A(G_h)$. Since all vertices $x,y,x_1,y_1,z_1,z_2$ have in-degree at least 1 in $G_h$, and since we did not change the direction of edges incident with other vertices of $G$, $G_h$ is a valid orientation of $G$. Since $\{x, y\}$ dominates four vertices in $G_h$, $\gt(G_h) \leq n(G)-2$, a contradiction.

Thus $N(x)=N(y)=\{x_1,x_2,x_3\}$. Since $G$ is not isomorphic to $K_{2,3}$ and is connected, there is either an edge between $x_i$ and $x_j$, for $i\neq j $ or some $x_i$ has a neighbor $z \not\in\{ x,y,x_j,x_k\}$, where $\{x_i,x_j,x_k\}=\{x_1,x_2,x_3\}$. If the latter holds, we may assume that $x_1$ has a neighbor $z \not\in\{ x,y,x_2,x_3\}$. Let $h$ be the orienting mapping of $G$ obtained from $f$ by forcing the following directions: $(x,x_1),(x,x_2),(y,x_2), (y,x_3), (x_1,y),(x_3,x), (x_1,z) \in A(G_h)$. Since all vertices in $\{x,y,x_1,x_2,x_3,z\}$ have in-degree at least 1 in $G_h$, and since we did not change the direction of edges incident with other vertices of $G$, $G_h$ is a valid orientation of $G$. Since $\{x, x_1\}$ dominates four vertices in $G_h$, $\gt(G_h) \leq n(G)-2$, a contradiction. Finally assume that without loss of generality $x_1x_2\in E(G)$. Furthermore we may without loss of generality assume that $(x_1,x_2) \in A(G_f)$. Let $h$ be the orienting mapping of $G$ obtained from $f$ by forcing the following directions: $(x,x_1),(x_2,x),(y,x_1), (y,x_3), (x_2,y),(x_1,x_2), (x,x_3) \in A(G_h)$. Since all vertices in $\{x,y,x_1,x_2,x_3\}$ have in-degree at least 1 in $G_h$, and since we did not change the direction of edges incident with other vertices of $G$, $G_h$ is a valid orientation of $G$. Since $\{x, x_2\}$ dominates four vertices in $G_h$, $\gt(G_h) \leq n(G)-2$, the final contradiction. 
 
\end{proof}

\begin{lem}\label{l:1degree3}
If $G$ is a connected graph with $\dt(G)=n(G)-1$, then $G$ is isomorphic to $K_{2,3},K_4,K_4-e$ or $G$ has at most one vertex of degree 3. 
\end{lem}

\begin{proof}
Let $G$ be a graph not isomorphic to $K_{2,3},K_4$, or $K_4-e$ and let $f$ be a valid orienting mapping of $G$.
By Lemma~\ref{l:degree3vertices} the vertices of degree 3 induce a complete graph. Since by Lemma~\ref{l:maxDegree} $G$ has no vertices of degree more than 3, $G$ contains at most four degree 3 vertices. If $G$ contains four degree 3 vertices, then they induce $K_4$ and as $G$ is connected with $\Delta(G)=3$, $G$ is isomorphic to $K_4$, a contradiction. 

Suppose now that $G$ has three degree 3 vertices $x,y$, and $z$. By Lemma~\ref{l:degree3vertices} they induce a $K_3$. Let $x_1,y_1,z_1$ (all three different from $x,y,z$) be neighbors of $x,y$ and $z$, respectively. Since $G$ has only three degree 3 vertices, $x_1=y_1=z_1$ is not possible. Without loss of generality we may assume that $x_1 \neq z_1$. Let $G_h$ be an orientation of $G$ with $(x,y),(y,z),(z,x),(x,x_1),(z,z_1)\in A(G_h)$ and all other edges of $G$ have the same direction in $G_h$ as in $G_f$. Since all vertices of $G_h$ have in-degree at least 1, $h$ is valid. Since vertices $\{x, z\}$ dominates four vertices in $G_h$, $\gt(G_h) \leq n(G)-2$, a contradiction. 

Finally assume that $G$ has exactly two vertices, say  $x$ and $y$, of degree 3. Since $G$ is not isomorphic to $K_4-e$, $N(x)\setminus \{y\} \neq N(y) \setminus \{x\}$. Without loss of generality let $(x,y) \in A(G_f)$. If $x$ and $y$ have a common neighbor in $G$, then let $x_1$ be their common neighbor, otherwise let $x_1$ be an in-neighbor of $x$ in $G_f$.
If $x_1$ is a common neighbor of $x$ and $y$, then let $h$ be an orienting mapping of $G$ such that $(x,y),(x_1,x),(y,x_1) \in A(G_h)$ and  all other edges incident to $x$ or $y$  are directed out of $x$ or $y$.  Moreover,  let all other edges of $G$ have the same direction in $G_h$ as in $G_f$. If $x$ and $y$ have no common neighbors, then let $h$ be an orienting mapping with $(x_1,x),(x,y) \in A(G_h)$ and all other edges incident to $x$ or $y$  are directed out of $x$ or $y$.  Moreover,  let all other edges of $G$ have the same direction in $G_h$ as in  $G_f$. Then $h$ is valid and since $\{x, y\}$ dominates four vertices in $G_h$, $\gt(G_h) \leq n(G)-2$, a contradiction.
Thus the only possibility is that $G$ has at most one vertex of degree 3.
\end{proof}

Before we state the main result of this section we define a family of graphs $\mathcal{F}$. A graph $G$ is in $\mathcal{F}$ if and only if $G$ is obtained from a cycle $C_k$, $k \geq 3$, and path $P_{\ell}$, $\ell \geq 2$, by identifying one of the leaves of the path with an arbitrary vertex of the cycle.

\begin{thm}\label{th:n-1}
Let $G\in \mathcal{C}$ be a connected graph. Then $\dt(G)=n(G)-1$ if and only if $G \in {\mathcal{F}} \cup \{K_4,K_{2,3},K_4-e\}$.
\end{thm}
\begin{proof}
If $G \in \{K_4,K_4-e,K_{2,3}\}$, then clearly $\dt(G)=n(G)-1$. Let $G \in \mathcal{F}$. Since $G$ is unicyclic, it follows from Proposition~\ref{unicyclic} that $\dt(G)=n(G)-L(G)=n(G)-1$.

For the converse let $G\in \mathcal{C}$ be a connected graph with $\dt(G)=n(G)-1$ different from $K_4,K_4-e$, and $K_{2,3}$. We first note, that $G$ is not a cycle, as cycles have lower orientable total domination number equal to their order. By Observation~\ref{tdsetcycle} $G$ contains a cycle. Let $C$ be a longest cycle in $G$. Since by Lemma~\ref{l:1degree3} $G$ has at most one vertex of degree 3, $C$ has no chords. As $G$ is connected and not isomorphic to a cycle, $C$ has a vertex $x$ of degree 3, which is the only vertex of degree 3 in $G$. If $H=G-(C-x)$ has a cycle, then either $G$ has a vertex of degree 4 or $G$ has a vertex of degree 3 different from $x$, which contradicts Lemma~\ref{l:maxDegree} or Lemma~\ref{l:1degree3}. Moreover, no vertex from $V(H)\setminus \{x\}$ has a neighbor on $C$, by Lemma~\ref{l:1degree3}.  Hence $H$ is a forest and $G$ is a unicyclic graph. Since $\dt(G)=n(G)-1$ it follows from Proposition~\ref{unicyclic} that $G$ has exactly one leaf. Thus $H$ is a path with end-vertex $x$ and therefore $G \in \mathcal{F}$.

\end{proof}

We end this section by posing the following problem. Note that if $\dt(G) = n(G) - 1$, then $\dt(G) = n(G)-1 = \Dt(G)$. Therefore, Theorem~\ref{th:n-1} is a partial characterization of the graphs with $\Dt(G) = n(G)-1$. However, we believe there exist graphs where $\dt(G) < n(G)-1 = \Dt(G)$. 

\begin{prob} Characterize all graphs where $\Dt(G) = n(G)-1$.
\end{prob}

\section*{Acknowledgments}

This work was performed within the bilateral grant ``Domination in graphs, digraphs and their products", Slovenia (BI-US/22-24-038).
T.D. was supported by the Slovenian Research and Innovation agency (grants P1-0297, J1-2452).


\end{document}